\documentclass{amsart}
\title{Nonregular ideals}
\author{Monroe Eskew}

\address{Universit\"at Wien \\
Institut f\"ur Mathematik \\
Kurt G\"odel Research Center \\
Augasse 2-6, UZA 1 - Building 2 \\
1090 Wien \\
AUSTRIA}

\date{}

\usepackage{color} 
\usepackage{amssymb} 
\usepackage{amsmath} 
\usepackage[utf8]{inputenc} 
\usepackage{enumerate}
\usepackage{amsthm}
\usepackage{cite}

\newtheorem{theorem}{Theorem}
\newtheorem{lemma}[theorem]{Lemma}
\newtheorem{proposition}[theorem]{Proposition}

\DeclareMathOperator{\dom}{dom}

\DeclareMathOperator{\ot}{ot}

\DeclareMathOperator{\cf}{cf}

\DeclareMathOperator{\add}{Add}

\DeclareMathOperator{\id}{id}

\newcommand{\p}{\mathcal{P}}
\newcommand{\la}{\langle}
\newcommand{\ra}{\rangle}

\subjclass[2010]{03E05, 03E35}

\keywords{ideals, regularity, forcing}

\begin{document}
\maketitle

\begin{abstract}
Generalizing Keisler's notion of regularity for ultrafilters, Taylor introduced degrees of regularity for ideals and showed that a countably complete nonregular ideal on $\omega_1$ must be somewhere $\omega_1$-dense.  We prove a dichotomy about degrees of regularity for $\kappa$-complete ideals on successor cardinals $\kappa$ and apply this to show that Taylor's Theorem does not generalize to higher cardinals.  In particular, the existence of a nonregular ideal on $\omega_2$ does not imply the existence of an $\omega_2$-dense ideal on $\omega_2$.  We obtain similar results for normal ideals on $\p_\kappa(\lambda)$.
\end{abstract}

 An \emph{ideal} on a set $X$ is a collection of subsets of $X$ closed under taking subsets and pairwise unions.  If $\kappa$ is a cardinal, an ideal $I$ is called \emph{$\kappa$-complete} if it is also closed under unions of size less than $\kappa$.  An ideal $I$ on $X$ is called \emph{nonprincipal} when for all $x \in X$, $\{x \} \in I$, and it is called \emph{proper} when $X \notin I$.  In this paper, we assume all our ideals are nonprincipal and proper.  An ideal $I$ on $X$ gives a notion of a ``negligible'' subset of $X$, and members of $I$ are called \emph{$I$-measure-zero}.  Subsets of $X$ which are not in $I$ are called \emph{$I$-positive}, and the collection of these is typically denoted by $I^+$.  The dual filter to $I$, the collection of all complements of members of $I$, constitutes the collection \emph{$I$-measure-one} sets and will be denoted by $I^*$.  If an ideal $I$ renders every subset of $X$ either measure zero or measure one, then its dual filter is called an \emph{ultrafilter}.
 

The notion of regularity of ultrafilters was introduced by Keisler \cite{keisler} and has had many applications in set theory and model theory \cite{ck}.  An ultrafilter $\mathcal U$ is called \emph{$(\alpha,\beta)$-regular} when there is a sequence of sets $\la A_i : i < \beta \ra \subseteq \mathcal U$ such that for all $z \subseteq \beta$ of ordertype $\alpha$, $\bigcap_{i \in z} A_i = \emptyset$.  Taylor \cite{t1} generalized this notion to arbitrary filters (or equivalently, ideals), defining an ideal $I$ to be $(\alpha,\beta)$-regular when for every sequence $\la A_i : i < \beta \ra \subseteq I^+$, there is a \emph{refinement} $\la B_i : i < \beta \ra \subseteq I^+$, which means $B_i \subseteq A_i$ for each $i$, such that for all $z \subseteq \beta$ of ordertype $\alpha$, $\bigcap_{i \in z} B_i = \emptyset$.  An ideal on a cardinal $\kappa$ is called simply \emph{regular} when it is $(\omega,\kappa)$-regular.  Taylor showed some connections between regularity properties of ideals and the structure of their associated quotient Boolean algebras, most notably the following:

\begin{theorem}[Taylor]
\label{omega1}
A countably complete ideal $I$ on $\omega_1$ is nonregular if and only if there is a set $A \in I^+$ such that $\p(A)/I$ contains a dense set of size $\omega_1$.
\end{theorem}

Taylor also discussed degrees of regularity indexed by three ordinals.  An ideal $I$ is said to be \emph{$(\alpha,\beta,\gamma)$-regular} when for every sequence $\la A_i : i < \gamma \ra \subseteq I^+$, there is a refinement $\la B_i : i < \gamma \ra \subseteq I^+$
such that for every $x \subseteq \gamma$ of ordertype $\beta$, $|\bigcap_{i \in x} B_i| \leq \alpha$.  We note the following easy relations between the regularity properties:
\begin{enumerate}
\item If $\alpha_0 < \alpha_1$, then $(\alpha_0,\beta,\gamma)$-regularity implies $(\alpha_1,\beta,\gamma)$-regularity.
\item If $\beta_0 < \beta_1$, then $(\alpha,\beta_0,\gamma)$-regularity implies $(\alpha,\beta_1,\gamma)$-regularity.
\item If $\gamma_0 < \gamma_1$, then $(\alpha,\beta,\gamma_1)$-regularity implies $(\alpha,\beta,\gamma_0)$-regularity.
\end{enumerate}


Taylor \cite{t1} showed that if $I$ is a $\kappa$-complete ideal on a regular uncountable cardinal $\kappa$, then $I$ is $(\omega,\kappa)$-regular if and only if it is $(2,\kappa)$-regular.  The latter is known as the \emph{disjoint refinement property} or \emph{Fodor's property} \cite{bhm}. 
In \cite{thesis}, the author showed that under GCH, many more degrees of regularity are equivalent for $\kappa$-complete ideals on $\kappa$, where $\kappa$ is the successor of a regular cardinal, and this was used to examine the relationship between regularity and density of ideals on cardinals above $\omega_1$.  In this paper as elsewhere, we only consider degrees of regularity of ideals on $\kappa$ for which the last index in the degree is at most $\kappa$.   Under this restriction, we show that without any assumptions, there are only two possible flavors of two-variable regularity at successor cardinals, and with GCH, only two possible flavors of three-variable regularity:

\begin{theorem}
\label{succ}
Suppose $\mu$ is an infinite cardinal, $\kappa = \mu^+$, and $I$ is a $\kappa$-complete ideal on $\kappa$.  Then:
\begin{enumerate}
\item\label{zfcreg} $I$ is $(\cf(\mu)+1,\kappa)$-regular, $(1,\cf(\mu),\kappa)$-regular, and $(2,\delta)$-regular for $\delta <\kappa$.
\item\label{regdich} If $I$ is $(\cf(\mu),\kappa)$-regular, then $I$ is $(2,\kappa)$-regular.
\item\label{gchreg} If $I$ is $(\alpha,\beta,\kappa)$-regular for some $\alpha,\beta < \kappa$ such that $\mu^\beta = \mu$, then $I$ is $(2,\kappa)$-regular.
\end{enumerate}
\end{theorem}

Part (\ref{zfcreg}) already appeared in \cite{t1} for the case of $\mu$ regular, and though the extension to the general case does not require essentially new ideas, we include a proof here for completeness.  We show similar results for $\kappa$-complete normal ideals on $\p_\kappa(\lambda)$.  

We will say a normal ideal on $Z \subseteq \p(\lambda)$ is simply \emph{regular} when it is $(2,\lambda)$-regular.  Following Taylor \cite{t1}, we show in Proposition \ref{regdef} that $(2,\lambda)$-regularity is equivalent to $(\omega,\lambda)$-regularity for such ideals, so this definition accords with the original terminology of Keisler.

%
%
%


It is easy to see that a $\lambda$-dense normal ideal on $\p(\lambda)$ is nonregular.  Taylor's Theorem uses a result of Baumgarter-Hajnal-M\'at\'e \cite{bhm}, who showed that if a countably complete ideal on $\omega_1$ is nowhere $\omega_1$-dense, then it has the disjoint refinement property.  This generalizes to normal ideals $I$ on $Z \subseteq \p_\kappa(\lambda)$ for $\kappa$ a successor cardinal, with an additional assumption about the quotient Boolean algebra $\p(Z)/I$ that is trivially satisfied for $\kappa = \omega_1$ (see \cite{thesis}).  However, it is possible to separate density and nonregularity above $\omega_1$:

\begin{theorem}
\label{dense}
Suppose $\kappa = \mu^+$, $\omega_1\leq \cf(\mu)$, $\kappa \leq \lambda$, and there is a nonregular, $\kappa$-complete, normal ideal on $\p_\kappa(\lambda)$.  There is a cardinal-preserving forcing extension that also has such an ideal, but in which there are no $\lambda$-dense, $\kappa$-complete, normal ideals on $\p_\kappa(\lambda)$.
\end{theorem}

By results in \cite{thesis}, the existence of a $\lambda$-dense, $\kappa$-complete, normal ideal on $\p_\kappa(\lambda)$, where $\kappa = \mu^+$, is consistent relative to an almost-huge cardinal, for any choice of regular cardinals $\mu<\kappa\leq\lambda$.

%

\section{The regularity dichotomy}

This section is devoted to a proof of Theorem \ref{succ}.  We will prove some more general facts about the regularity of normal ideals on $\p(\lambda)$ and show how they imply the desired results about $\kappa$-complete ideals on successor cardinals $\kappa$.

Our notations are mostly standard.  By $\p_\kappa(\lambda)$ we mean $\{ z \subseteq \lambda : |z| < \kappa \}$.  If $x$ is a set of ordinals, then $\ot(x)$ denotes its ordertype.  The notations $[\lambda]^\kappa$ and $[\lambda]^{<\kappa}$ stand for $\{ z \subseteq \lambda : \ot(z) = \kappa \}$ and $\{ z \subseteq \lambda : \ot(z) < \kappa \}$ respectively.

The following facts can be found in \cite{foremanhandbook}.  Recall that an ideal $I$ on $Z \subseteq \p(X)$ is \emph{normal} when for all $x \in X$, $\hat x := \{ z \in Z : x \in z \} \in I^*$, and for all sequences $\la A_x : x \in X \ra \subseteq I$, the diagonal union $\nabla_{x \in X} A_x := \bigcup_{x \in X} (A_x \cap \hat x) \in I$.  This is equivalent to the statement that for every $A \in I^+$ and every $f : A \to X$ such that $f(z) \in z$ for all $z \in A$, there is $B \in I^+$ such that $f$ is constant on $B$.

The smallest normal ideal on a set $Z$ is the \emph{nonstationary ideal on $Z$}, which is the dual ideal to the \emph{club filter} (closed-unbounded filter) generated by sets of the form $\{ z \in Z : f[z^{<\omega}] \subseteq z \}$, where $f$ is a function from $X^{<\omega}$ to $X$.  As the name suggests, positive sets for the nonstationary ideal are called \emph{stationary}.  Consequently, if there is a (proper) normal ideal on $Z \subseteq \p(X)$, then $Z$ is stationary.  


A normal ideal $I$ on $Z \subseteq \p(X)$ is \emph{$\delta$-saturated} for a cardinal $\delta$ if there is no sequence $\la A_\alpha : \alpha < \delta \ra$ such that $A_\alpha \cap A_\beta \in I$ for $\alpha < \beta$, and simply \emph{saturated} if it is $|X|^+$-saturated.  If $I$ is saturated, then $\p(Z)/I$ is a complete Boolean algebra, with suprema given by diagonal unions.   If $\la A_x : x \in X \ra$ is an antichain, then we can use normality to refine it to a pairwise disjoint sequence of $I$-positive sets by replacing $A_x$ with $A_x \cap \hat x \setminus \bigcup_{y \not= x} (A_y \cap \hat y)$.


The idea behind the following lemma is taken from \cite{bhm}.

\begin{lemma}
\label{split}
Suppose $I$ is an ideal on $Z \subseteq \p(\lambda)$, $\delta \leq \lambda$, and $I$ is either normal or $\delta$-complete.  If there is no $A \in I^+$ such that $I \restriction A$ is $\delta^+$-saturated, then $I$ is $(2,\delta)$-regular.
\end{lemma}
\begin{proof}
Let $\la A_\alpha : \alpha < \delta \ra \subseteq I^+$, and for each $A_\alpha$, choose a sequence of $I$-positive sets $\la B_\alpha^\beta : \beta < \delta^+ \ra$ such that each $B_\alpha^\beta \subseteq A_\alpha$ and $B_\alpha^\beta \cap B_\alpha^{\beta'} \in I$ when $\beta < \beta' < \delta^+$.  For each $\alpha < \delta$, let $f(\alpha) \leq \alpha$ be the minimal ordinal such that $| \{ \beta : A_\alpha \cap B_{f(\alpha)}^\beta \in I^+ \} | = \delta^+$.
We can find $\xi < \delta^+$ such that for all $\alpha < \delta$, all $\alpha' < f(\alpha)$, and all $\beta \geq \xi$, $A_\alpha \cap B^\beta_{\alpha'} \in I$.

Recursively choose a refinement $\la C_\alpha : \alpha < \delta \ra$ of $\la A_\alpha : \alpha < \delta \ra$ and an increasing sequence of ordinals $\la \beta_\alpha : \alpha < \delta \ra \subseteq\delta^+$ as follows.  Let $C_0 = B_0^\xi$ and $\beta_0 = \xi$.  Given $\la C_{\alpha'} : \alpha' < \alpha \ra$, let $C_\alpha$ be an $I$-positive set of the form $A_\alpha \cap B_{f(\alpha)}^{\beta_\alpha}$, where $\beta_\alpha \geq \sup_{\alpha' < \alpha}( \beta_{\alpha'}+1)$.  Note that whenever $\alpha \not= \alpha'$ are less than $\delta$, it is ensured that $C_{\alpha} \cap C_{\alpha'} \in I$.  This is because if $f(\alpha) = f(\alpha') = \eta$, then $B^{\beta_{\alpha'}}_\eta \cap B^{\beta_{\alpha}}_\eta \in I$ by construction, and if $f(\alpha) < f(\alpha')$, then $B^{\beta_\alpha}_{f(\alpha)} \cap A_{\alpha'} \in I$ since $\beta_\alpha \geq \xi$.

Finally, we refine $\la C_\alpha : \alpha < \delta \ra$, to a pairwise disjoint sequence $\la D_\alpha : \alpha < \delta \ra$.  If $I$ is normal, we put $D_\alpha = C_\alpha \cap \hat \alpha \setminus \bigcup_{\alpha' \not= \alpha} (C_{\alpha'} \cap {\hat\alpha}^\prime)$.  If $I$ is $\delta$-complete, we put $D_\alpha = C_\alpha \setminus \bigcup_{\alpha'<\alpha} C_{\alpha'}$.
\end{proof}

\begin{proposition}
\label{regdef}
A normal ideal on $\p(\lambda)$ is $(\omega,\lambda)$-regular if and only if it is $(2,\lambda)$-regular.
\end{proposition}
\begin{proof}
Suppose $I$ is an $(\omega,\lambda)$-regular normal ideal on $\p(\lambda)$.  Let $\la A_\alpha : \alpha < \lambda \ra \subseteq I^+$.  We may assume that $A_\alpha \subseteq \hat \alpha$ for each $\alpha<\lambda$.  Let $\la B_\alpha : \alpha < \lambda \ra \subseteq I^+$ be a refinement such that $\bigcap_{\alpha \in x} B_\alpha = \emptyset$ whenever $x \subseteq \lambda$ is infinite.  For each $z \in \bigcup_{\alpha<\lambda} B_\alpha$, let $s(z)$ be the finite set $\{ \alpha : z \in B_\alpha \}$.  Note that $s(z) \subseteq z$.

Using the normality of $I$, for each $\alpha<\lambda$, there is an $I$-positive $C_\alpha \subseteq B_\alpha$ such that $s$ is constant on $C_\alpha$ with value $g(\alpha) \in [\lambda]^{<\omega}$.  Note that if $g(\alpha) \not= g(\beta)$, then $C_\alpha \cap C_\beta = \emptyset$.

Now since $I$ is $(\omega,\lambda)$-regular and countably complete, there is no $A \in I^+$ such that the dual of $I \restriction A$ is a ultrafilter.  Therefore every $I$-positive set can be partitioned into two disjoint $I$-positive sets, and thus infinitely many.  Hence for every $x \in [\lambda]^{<\omega}$, the sequence $\la C_\alpha : \alpha \in x \ra$ has a disjoint refinement $\la D^x_\alpha : \alpha \in x \ra \subseteq I^+$ by Lemma \ref{split}.  Finally, let $E_\alpha = D^{g(\alpha)}_\alpha$.  Then $\la E_\alpha : \alpha < \lambda \ra$ is a pairwise disjoint refinement of $\la A_\alpha : \alpha<\lambda \ra$.
\end{proof}

The above argument can be generalized to show that for a $\kappa$-complete normal ideal on $\p(\lambda)$ and $\mu<\kappa$, $(\mu,\lambda)$-regularity implies $(2,\lambda)$-regularity if the ideal concentrates on $z  \subseteq \lambda$ such that $|z|^{<\mu} = |z|$.  However, we ultimately want to prove something more general about successor $\kappa$.  The following lemma contains the key combinatorial idea.

\begin{lemma}
\label{main}
Suppose $I$ is a normal ideal on $\p(\lambda)$, $\mu$ is a cardinal such that $\{ z \subseteq \lambda : \cf(\sup z) \geq \mu \} \in I^*$, and for all $A \in I^+$ and $\delta < \lambda$, $I \restriction A$ is not $\delta^+$-saturated.  If $I$ is $(\mu,\lambda)$-regular, then $I$ is regular.
\end{lemma}

\begin{proof}
Let $\la A_\alpha : \alpha < \lambda \ra \subseteq I^+$.  We may assume that for all $\alpha <\lambda$ and all $z \in A_\alpha$,  $\cf(\sup z) \geq \mu$ and $\alpha \in z$.  Since $I$ is $(\mu,\lambda)$-regular, there is a sequence of $I$-positive sets $\la B_\alpha : \alpha < \lambda \ra $ such that $B_\alpha \subseteq A_\alpha$ for all $\alpha$, and for all $z$, $s(z) := \{ \alpha : z \in B_\alpha \}$ has size $<\mu$.  Note that $s(z) \subseteq z$.  For all $z \in \bigcup_{\alpha<\lambda}B_\alpha$, since $\cf(\sup z) \geq \mu$, $s(z)$ is not cofinal in $z$.  Thus let $f : \bigcup_{\alpha<\lambda} B_\alpha \to \lambda$ be such that $s(z) \subseteq f(z) \in z$.  By normality, for all $\alpha<\lambda$, there is an $I$-positive $C_\alpha \subseteq B_\alpha$ on which $f$ is constant.  Let $g(\alpha)$ be this constant value, and note that $g(\alpha) > \alpha$.

For each $\alpha < \lambda$, choose a pairwise disjoint refinement $\langle D^\alpha_\beta : \beta < \alpha \rangle \subseteq I^+$ of $\langle C_\beta : \beta < \alpha \rangle$, using Lemma \ref{split}.  Then let $E_\alpha = D_\alpha^{g(\alpha)}$.  If $g(\alpha_0) = g(\alpha_1)$, then $E_{\alpha_0} \cap E_{\alpha_1} = \emptyset$ by construction.   If $g(\alpha_0) \not= g(\alpha_1)$, then $E_{\alpha_0} \cap E_{\alpha_1} = \emptyset$, since for $i<2$ and $z \in E_{\alpha_i}$, $f(z) = g(\alpha_i)$.
\end{proof}

\begin{lemma}
\label{upperlim}
Suppose $\mu,\lambda$ are regular cardinals, and $I$ is a normal ideal on $\p(\lambda)$ such that $\{ z \subseteq \lambda : \cf(\sup z) = \mu \} \in I^*$.
Then $I$ is $(\mu+1,\lambda)$-regular.  If the function $z \mapsto \sup z$ is $\leq \delta$ to one on a set in $I^*$, then $I$ is $(\delta,\mu,\lambda)$-regular.
\end{lemma}

\begin{proof}
Let $\la A_i : i < \lambda \ra \subseteq I^+$.  Let $Z = \{ z \subseteq \lambda : \cf(\sup z) = \mu \}$.
For each $z \in Z$, let $c_z \subseteq z$ be a cofinal subset of ordertype $\mu$.    By induction, we build an increasing sequence $\la \alpha_i : i < \lambda \ra \subseteq \lambda$ and a refinement $\la B_i : i < \lambda \ra \subseteq I^+$ of $\la A_i : i < \lambda \ra$  as follows.  Given $\la \alpha_i : i < j \ra$, $\sup z > \sup_{i<j} \alpha_i$ for $I$-almost all $z \in A_j$. For such $z$, let $\sup_{i<j} \alpha_i < \alpha_j(z) \in c_z$.  Let $B_j \subseteq A_j$ be an $I$-positive set on which the function $z \mapsto \alpha_j(z)$ is constant, and let $\alpha_j$ be this constant value.  For each $z$, let $s(z) = \{ i < \lambda : z \in B_i \}$.  Note that $z \in B_i$ implies $\alpha_i \in c_z$, so $\ot(s(z)) \leq \mu$.  This establishes the claim that $I$ is $(\mu+1,\lambda)$-regular.  For the second claim, note that if $\ot(s(z)) = \mu$, then $s(z)$ is cofinal in $c_z$ and thus in $z$.  Thus, if $z \mapsto \sup z$ is $\leq \delta$ to one on a set in $I^*$, then we may take the sequence $\la B_i : i < \lambda \ra$ such that 
$|\bigcap_{i \in x} B_i | \leq \delta$ whenever $\ot(x) = \mu$.
\end{proof}



%

The following result was independently observed by Burke-Matsubara \cite{bm-strength} and Foreman-Magidor \cite{fm-mutual}.  Its proof uses deep results of Shelah \cite{shelahproper} and Cummings \cite{cummingscollapse}.

\begin{lemma}
\label{cofcon}
Suppose $I$ is a normal saturated ideal on $\p(\lambda)$. Then $\{ z : \cf(\sup z) = \cf(|z|) \} \in I^*$.
\end{lemma}

The following basic fact can be proved in multiple ways, for example via Ulam matrices or via generic ultrapowers (see \cite{foremanhandbook}).

\begin{lemma}
\label{nonsat}
If $\kappa$ is a successor cardinal, then no $\kappa$-complete ideal on $\kappa$ is $\kappa$-saturated, and no $\kappa$-complete normal ideal on $\p_\kappa(\lambda)$ is $\lambda$-saturated.
\end{lemma}

\begin{theorem}
\label{succgen}
Suppose $\kappa = \mu^+$ and $I$ is a $\kappa$-complete normal ideal on $\p_\kappa(\lambda)$.  If $I$ is $(\cf(\mu),\lambda)$-regular, then $I$ is regular.  If $\lambda$ is a regular cardinal, then $I$ is $(\cf(\mu)+1,\lambda)$-regular.
\end{theorem}

\begin{proof}
Let $\la A_\alpha : \alpha < \lambda \ra \subseteq I^+$.  We first separate the saturated and non-saturated parts.  We choose an initial refinement by putting $B_\alpha = A_\alpha$ if there is no $B \subseteq A_\alpha$ such that $I \restriction B$ is saturated, and otherwise choose $B_\alpha \subseteq A_\alpha$ such that $I \restriction B_\alpha$ is saturated.  Let $Y_0$ be the ordinals below $\lambda$ falling into the first case, and $Y_1$ those falling into the second.  Note that whenever $\alpha \in Y_0$ and $\beta \in Y_1$, we have $B_\alpha \cap B_\beta \in I$.  As in the proof of Lemma \ref{split}, we may refine to a sequence $\la C_\alpha : \alpha < \lambda \ra$ such that $C_\alpha \cap C_\beta = \emptyset$ whenever at least one of $\alpha,\beta$ is in $Y_0$.  
If we put $C =  \nabla_{\alpha \in Y_1} C_\alpha$, then $I \restriction C$ is saturated, since if $\la D_\alpha : \alpha < \lambda^+ \ra$ were an antichain in $\p(C)/I$, then for some $\beta < \lambda$, there would be $\lambda^+$-many $\alpha$ such that $C_\beta \cap D_\alpha \in I^+$, in contradiction to the fact that $I \restriction C_\beta$ is saturated.

We may assume $\cf(\sup z) = \cf(\mu)$ for all $z \in C$.   Since $I \restriction A$ is not $\lambda$-saturated for any $A \in I^+$, Lemma \ref{main} implies that if $I$ is $(\cf(\mu),\lambda)$-regular, then there is a disjoint refinement of $\la C_\alpha : \alpha \in Y_1 \ra$ into $I$-positive sets $\la D_\alpha : \alpha \in Y_1 \ra$.  Putting this together with $\la C_\alpha : \alpha \in Y_0 \ra$, we have a disjoint refinement of the original sequence into $I$-positive sets.

If $\lambda$ is a regular cardinal, then by Lemma \ref{upperlim}, there is a refinement $\la E_\alpha : \alpha \in Y_1 \ra \subseteq I^+$ of $\la C_\alpha : \alpha \in Y_1 \ra$ such that $\bigcap_{\alpha \in x} E_\alpha = \emptyset$ whenever $\ot(x) > \cf(\mu)$, showing that $I$ is $(\cf(\mu) + 1,\lambda)$-regular. 
\end{proof}

In order to prove Theorem \ref{succ}, we use some results from \cite{t1} which allow a reduction to normal ideals:

\begin{lemma}[Taylor]
\label{taylor}
Let $I$ be a $\kappa$-complete ideal on $\kappa$.
\begin{enumerate}
\item Suppose every sequence $\la A_i : i< \kappa\ra \subseteq I^+$ has a refinement $\la B_i : i< \kappa\ra \subseteq I^+$ such that $I \restriction B_i$ is $(\alpha,\beta,\kappa)$-regular for each $i$.  Then $I$ is $(\alpha,\beta,\kappa)$-regular.
\item If $\kappa= \mu^+$ and $I$ is $\kappa^+$-saturated, then there is $A \in I^+$ and a bijection $f :  \kappa \to \kappa$ such that $\{ f[X] : X \in I \restriction A \}$ is a normal ideal on $\kappa$.
\end{enumerate}
\end{lemma}

Let $I$ be a $\kappa$-complete ideal on $\kappa = \mu^+$.  First let us show part (\ref{zfcreg}) of Theorem \ref{succ}.  By Lemmas \ref{split} and \ref{nonsat}, $I$ is $(2,\delta)$ regular for $\delta < \kappa$.
 For the other regularity properties, let $\la A_\alpha : \alpha < \kappa \ra \subseteq I^+$.  Let $B_\alpha \subseteq A_\alpha$ be an $I$-positive set such that $I \restriction B_\alpha$ is $\kappa^+$-saturated if there is such a $B_\alpha$.  In such a case, part (2) of Lemma \ref{taylor} implies that we can find an $I$-positive $C_\alpha \subseteq B_\alpha$ such that $I \restriction C_\alpha$ is isomorphic to a normal ideal.    
By Lemmas \ref{upperlim} and \ref{cofcon}, $I \restriction C_\alpha$ is $(\cf(\mu)+1,\kappa)$-regular and $(1,\cf(\mu),\kappa)$-regular whenever $C_\alpha$ is defined.  If $C_\alpha$ is undefined, then $I \restriction A_\alpha$ is regular by Lemma \ref{split}.  Part (1) of Lemma \ref{taylor} then gives that $I$ is $(\cf(\mu)+1,\kappa)$-regular and $(1,\cf(\mu),\kappa)$-regular.  

Now let us show part (\ref{regdich}).  Let $\la A_\alpha : \alpha < \kappa \ra \subseteq I^+$, and choose sets $C_\alpha$ exactly as above.  If $I$ is $(\cf(\mu),\kappa)$-regular, then so is each $I \restriction C_\alpha$ when $C_\alpha$ is defined, and thus $I \restriction C_\alpha$ is regular by Theorem \ref{succgen}.  Again by Lemma \ref{split} and part (1) of Lemma \ref{taylor}, $I$ is regular in this case.  


To show part (\ref{gchreg}) of Theorem \ref{succ}, we introduce an extension of Taylor's three-variable notion of regularity.  Let us say an ideal $I$ is \emph{$(I,\alpha,\beta)$-regular} if every sequence $\la A_i : i < \beta \ra \subseteq I^+$ has a refinement $\la B_i : i < \beta \ra \subseteq I^+$ such that $\bigcap_{i \in x} B_i \in I$ whenever $\ot(x) = \alpha$.  If $I$ is a $\kappa$-complete ideal on $\kappa$, then $(I,\beta,\kappa)$-regularity is a weakening of $(\alpha,\beta,\kappa)$-regularity for every $\alpha < \kappa$.

\begin{lemma}
\label{measreg}
Suppose $\kappa=\mu^+$ and $I$ is a $\kappa$-complete ideal on $\kappa$.  If $I$ is $(I,\xi,\kappa)$-regular, where $\mu^\xi = \mu$, then $I$ is regular.
\end{lemma}

\begin{proof}
Let $\la A_\alpha : \alpha < \kappa \ra \subseteq I^+$, and let $\la B_\alpha : \alpha < \kappa \ra \subseteq I^+$ be a refinement such that $B_\alpha \subseteq \hat \alpha$ for all $\alpha$, and $\bigcap_{\alpha \in x} B_\alpha \in I$ whenever $\ot(x) \geq \xi$.  For every $\alpha < \kappa$ we can define an $I$-positive $C_\alpha \subseteq B_\alpha$ by 
$$C_\alpha = B_\alpha \setminus \bigcup_{x \in [\alpha]^\xi} \bigcap_{\beta\in x} B_\beta.$$
If $x$ is a subset of $\kappa$ of ordertype $\xi + 1$, then let $\alpha = \max(x)$.  If $\beta \in C_\alpha$, then $\beta \notin \bigcap_{\gamma \in x\cap\alpha} C_\gamma$.  This shows that $I$ is $(\xi+1,\kappa)$-regular.  Since $\mu^\xi = \mu$ implies $\xi<\cf(\mu)$, $I$ is regular by part (\ref{regdich}) of Theorem \ref{succ}.
\end{proof}

\section{Consistency results}
This section is devoted to a proof of Theorem \ref{dense}.  If $V \subseteq W$ are models of set theory and $I \in V$ is an ideal, then in $W$ we can generate an ideal $\bar I$ from $I$ by taking all sets which are covered by a set from $I$.  Let us first show the preservation of nonregular ideals by forcings with a strong enough chain condition, as a consequence of Theorem \ref{succgen}.

\begin{lemma}
\label{preservenonreg}
Suppose $\kappa = \mu^+$, $\lambda \geq \kappa$, and $I$ is a nonregular, $\kappa$-complete, normal ideal on $Z \subseteq \p_\kappa(\lambda)$.  If $\mathbb P$ is $\cf(\mu)$-c.c., then in $V^{\mathbb P}$, the ideal $\bar I$ generated by $I$ is nonregular.
\end{lemma}
\begin{proof}
Let us show the contrapositive, that if $\bar I$ is regular in a $\mathbb P$-generic extension, then $I$ is regular in $V$.
Let $\la A_\alpha : \alpha <\lambda \ra \subseteq I^+$ be in $V$.  If $p \Vdash \bar I$ is regular, then there is a $\mathbb P$-name for a refinement $\la \dot B_\alpha : \alpha <\lambda \ra$ such that each $z \in Z$ is forced by $p$ to be in 
at most one $B_\alpha$.
In $V$, for each $\alpha$ let $C_\alpha = \{ z \in A_\alpha : (\exists q \leq p)q \Vdash z \in \dot B_\alpha \}$.  Since $p \Vdash \dot B_\alpha \subseteq \check C_\alpha$, each $C_\alpha$ is $I$-positive.  By the chain condition, for each $z$, the set $s(z) := \{ \alpha : (\exists q \leq p)q \Vdash z \in \dot B_\alpha \} =  \{ \alpha : z \in C_\alpha \}$ has size $<\cf(\mu)$.  This shows that $I$ is $(\cf(\mu),\lambda)$-regular in $V$, and thus regular by Theorem \ref{succgen}.
\end{proof}

If $I$ is a $\kappa$-complete normal ideal and $\mathbb P$ is a $\kappa$-c.c.\ forcing, then it is easy to show that the ideal generated by $I$ is also $\kappa$-complete and normal in $V^{\mathbb P}$.  If $I$ is saturated, then Foreman's Duality Theorem \cite{foremanduality} allows us to say much more.  This is connected to the forcing properties of the quotient algebra and generic elementary embeddings.

The following facts can be found in\cite{foremanhandbook}.
If $I$ is an ideal on $Z$ and $G \subseteq \p(Z)/I$ is generic, then in $V[G]$, we can form the ultrapower embedding $j : V \to V^Z/G$.  If $Z \subseteq \p(\lambda)$ and $I$ is normal, then the pointwise image of $\lambda$ under $j$ is represented in the ultrapower by the identity function on $Z$, i.e.\ $[\id]_G = j[\lambda]$.  If $I$ is $\kappa$-complete, $\kappa = \mu^+$, and $Z \subseteq \p_\kappa(\lambda)$, then $\kappa$ is the critical point of $j$, and $V^Z/G \models |j[\lambda]| < j(\kappa)$.  Consequently, $V[G] \models |\lambda| = |\mu|$.  This implies that there is no condition $A \in I^+$ such that $I \restriction A$ is $\lambda$-saturated.
Thus in this context, $I$ being saturated is the same as $\p(Z)/I$ having the best possible chain condition.  If this occurs, then $I$ is \emph{precipitous}, meaning that whenever $G \subseteq \p(Z)/I$ is generic, $V^Z/G$ is well-founded and thus isomorphic to a transitive class $M \subseteq V[G]$.  



\begin{theorem}[Foreman \cite{foremanduality}]
\label{duality}
Suppose $I$ is a $\kappa$-complete precipitous ideal on $Z$, and $\mathbb P$ is a $\kappa$-c.c.\ forcing.  In $V^{\mathbb P}$, let $\bar I$ denote the ideal generated by $I$, and let $j$ denote a generic ultrapower embedding obtained from forcing with $\p(Z)/I$.  Then there is an isomorphism
$$\iota : \mathcal B(\mathbb P * \dot{\p(Z)} /\bar I) \cong \mathcal B(\p(Z)/I * \dot{j(\mathbb P)})$$
given by $\iota(p,\dot A) = || [\id] \in j(\dot A)|| \wedge (1,\dot{j(p)})$.
\end{theorem}

The next proposition shows the relevance of the cardinal arithmetic assumption in Lemma \ref{measreg}.  For example, we can produce a model in which CH fails and there is a nonregular ideal $I$ on $\omega_2$ which is $(I,\omega,\omega_2)$-regular.

\begin{proposition}Suppose $\kappa=\mu^+$, $\nu \leq \mu$ is such that $\nu^{<\nu} = \nu$, and $I$ is a saturated, nonregular, $\kappa$-complete ideal on $\kappa$.  If $G \subseteq \add(\nu,\kappa)$ is generic, then in $V[G]$, $\bar I$ is $(\bar I,\nu,\kappa)$-regular.
\end{proposition}

\begin{proof}

Since $\nu^{<\nu} = \nu$, $\add(\nu,\kappa)$ is $\nu^+$-c.c.  By Theorem \ref{duality}, in $V[G]$, there is an isomorphism $\sigma : \p(\kappa)/\bar I \cong \mathcal B(\p(\kappa)^V/I \times \add(\nu,\kappa^+))$.  If $\la A_\alpha : \alpha < \kappa \ra \subseteq\bar I^+$, choose for each $\alpha$ some $(B_\alpha,p_\alpha) \leq \sigma(A_\alpha)$. Let $\beta<\kappa^+$ be such that $\dom p_\alpha \subseteq \beta \times \nu$ for all $\alpha$.  Let $q_\alpha = \{ ((\beta+\alpha, 0),0) \}$ for $\alpha<\kappa$, and choose $C_\alpha \leq \sigma^{-1}(B_\alpha,p_\alpha\wedge q_\alpha)$. The intersection of any $\nu$-many $C_\alpha$ is in $\bar I$, since there is no lower bound to $\nu$-many $q_\alpha$.
\end{proof}

\begin{lemma}
\label{rest}
Suppose $I$ is a normal ideal on $Z \subseteq \p(X)$.  Then $I$ is $|X|^+$-saturated if and only if every normal $J \supseteq I$ is equal to $I \restriction A$ for some $A \subseteq Z$.
\end{lemma}

\begin{proof}
Suppose $I$ is $|X|^+$-saturated.  Let $Y \subseteq X$ and $\{ A_x : x \in Y \}$ be such that $\{ [A_x]_I : x \in Y \}$ is an antichain in $\p(Z)/I$ of $I$-positive sets that are $J$-measure-zero, and is maximal among all such collections.  Then $[\nabla A_x]_I$ is the $\subseteq_I$-largest element of $J \cap I^+$, so $J = I \restriction (Z \setminus \nabla A_x)$.  Now suppose $I$ is not $|X|^+$-saturated, and let $\{ A_\alpha : \alpha < \delta \}$ be a maximal antichain where $\delta \geq |X|^+$.  Let $J$ be the ideal generated by $\bigcup \{ \Sigma_{\alpha \in Y} [A_\alpha] : Y \in \p_{|X|^+}(\delta) \}$.  Then $J$ is a proper normal ideal extending $I$.  $J$ cannot be equal $I \restriction A$ for some $A \in I^+$ because if this were so, there would some $\alpha$ such that $A \cap A_\alpha \in I^+$.  $A \cap A_\alpha \in J$ by construction, but every $I$-positive subset of $A$ is $(I \restriction A)$-positive. 
\end{proof}

A partial order is said to be \emph{$\kappa$-dense} if it has a dense subset of size $\leq \kappa$.  It is said to be \emph{nowhere $\kappa$-dense} if it is not $\kappa$-dense below any condition.  An ideal is said to be $\lambda$-dense or nowhere $\lambda$-dense when its associated Boolean algebra has these properties.  

\begin{lemma}
\label{nodense}
Suppose $\mu^+=\kappa\leq\lambda$, $\nu \leq \mu$ is such that $\nu^{<\nu} = \nu$, and $Z \subseteq \p_\kappa(\lambda)$ is stationary.  Let $\mathbb P = \add(\nu,\theta)$ for some $\theta \geq \kappa$.  Then in $V^\mathbb{P}$, there are no normal, $\kappa$-complete, $\lambda$-dense ideals on $Z$.
\end{lemma}

\begin{proof}Since a $\lambda$-dense ideal is $\lambda^+$-saturated, it suffices to consider $\mathbb P$-names for $\lambda^+$-saturated normal ideals on $Z$.
Suppose $p \Vdash \dot{J}$ is a $\kappa$-complete, $\lambda^+$-saturated, normal ideal on $Z$.  Let $I = \{ X \subseteq Z : p \Vdash X \in \dot{J} \}$.  It is easy to check that $I$ is normal and $\kappa$-complete.  The map $\sigma : \p(Z) / I \to \mathcal{B}(\mathbb{P} \restriction p) * \p(Z) / \dot{J}$ that sends $X$ to $( || \check{X} \in \dot{J}^+ || , \dot{[X]_J} )$ is order-preserving and antichain-preserving.  Since $\nu^{<\nu} = \nu$, $\mathbb P$ is $\kappa$-c.c., so the two-step iteration $\mathbb{P} \restriction p * \p(Z) / \dot{J}$ is $\lambda^+$-c.c.  Thus $I$ is $\lambda^+$-saturated.

Let $H$ be $\mathbb{P}$-generic over $V$ with $p \in H$.  Since $\mathbb{P}$ is $\kappa$-c.c., $\bar{I}$ remains normal.  By Theorem~\ref{duality}, the map $e : q \mapsto (1,\dot{j(q)})$ is a regular embedding of $\mathbb P$ into $\p(Z)/I * \dot{j(\mathbb P)}$.  Thus in $V[H]$, $\p(Z) / \bar{I} \cong \p^V (Z) / I * \add(\nu,\dot \eta)$, where $\Vdash \dot \eta = \ot(j(\theta) \setminus j[\theta])$.  By the saturation of $I$, $\Vdash j(\kappa) = \lambda^+$, so $\Vdash \dot \eta \geq \lambda^+$.

$\bar{I}$ is normal and $\lambda^+$-saturated, and $\bar{I} \subseteq J$.  By Lemma~\ref{rest}, there is $A \in \bar{I}^+$ such that $J = \bar{I} \restriction A$.  
Since $\add(\nu,\eta)$ is nowhere $\lambda$-dense, $\p(Z)/\bar{I}$ is nowhere $\lambda$-dense.  Thus $J$ is not $\lambda$-dense. 
\end{proof}

%

Thus we may rid the universe of dense ideals that concentrate on $\p_\kappa(\lambda)^V$.  This finishes the job if $\kappa = \lambda$, but not necessarily in other cases.  For example, Gitik  \cite{gitiknonsplitting} showed that if $V \subseteq W$ are models of set theory, $\kappa < \lambda$ are regular in $W$, and there is a real number in $W \setminus V$, then $\p_\kappa(\lambda)^W \setminus \p_\kappa(\lambda)^V$ is stationary.   In order to take care of such problems, we use some arguments of Laver and Hajnal-Juhasz that are reproduced in \cite{foremanhandbook}.


The notation
$\left( \begin{array}{c}
\alpha \\ \beta
\end{array} \right)
\rightarrow
\left( \begin{array}{c}
\gamma \\ \delta
\end{array} \right)_\eta$
stands for the assertion that for every $f : \alpha \times \beta \to \eta$, there is $A \in [\alpha]^\gamma$ and $B \in [\beta]^\delta$ such that $f$ is constant on $A\times B$.  As usual with arrow notations, if ordinals on the left side are increased and ordinals on the right side are decreased, then we get a weaker statement.

\begin{lemma}
\label{densepart}
Suppose there is a $\lambda$-dense, $\kappa$-complete, normal ideal $I$ on $\p_\kappa(\lambda)$ such that every $I$-positive set has cardinality $\geq \eta$.  Then for $\mu,\nu < \kappa$,
$$\left( \begin{array}{c}
\lambda^+ \\ \lambda^{<\kappa}
\end{array} \right)
\rightarrow
\left( \begin{array}{c}
\mu \\ \eta
\end{array} \right)_\nu.$$
\end{lemma}

\begin{proof}
Let $\theta = \lambda^{<\kappa}$, and enumerate $\p_\kappa(\lambda)$ as $\la z_\alpha : \alpha < \theta \ra$.  Let $f : \lambda^+ \times \theta \to \nu$.  By $\kappa$-completeness, for each $\alpha < \lambda^+$, there is $\gamma < \nu$ such that $X_\alpha :=
\{ z_\beta : f(\alpha,\beta) = \gamma \} \in I^+$.  By $\lambda$-density, there is a set $S \in [\lambda^+]^{\lambda^+}$, a set 
 $D \in I^+$, and a $\gamma^* < \nu$ such that for all $\alpha \in S$, $D \subseteq_I X_\alpha$ and $f(\alpha,\beta) = \gamma^*$ for $z_\beta \in X_\alpha$.  Let $A \subseteq S$ have size $\mu$.  Since $\bigcap_{\alpha \in A} X_\alpha$ is $I$-positive, there is a set $B \subseteq \theta$ of size $\geq \eta$ such that for all $\alpha \in A$ and all $\beta \in B$, $f(\alpha,\beta) = \gamma^*$.
\end{proof}

\begin{lemma}
\label{killpart}
Suppose $\theta$ is regular and $\mu < \theta$ is such that $\mu^{<\mu} = \mu$.  If $G \subseteq \add(\mu,\theta)$ is generic, then in $V[G]$, 
$$\left( \begin{array}{c}
\theta^+ \\ \theta
\end{array} \right)
\nrightarrow
\left( \begin{array}{c}
\mu \\ \theta
\end{array} \right)_2.$$
\end{lemma}

\begin{proof}
In $V$, choose an almost-disjoint family $\{ X_\alpha : \alpha < \theta^+ \} \subseteq \p(\theta)$, and for each $\alpha$, let $\la \gamma^\alpha_\beta : \beta < \theta \ra$ enumerate $X_\alpha$ in increasing order.  In $V[G]$, let $f : \theta^+ \times \theta \to 2$ be defined by $f(\alpha,\beta) = G(\gamma^\alpha_\beta,0)$.  Let $A \subseteq \theta^+$ be a set of size $\mu$ in $V[G]$.  By the chain condition, there is a $\zeta <\theta$ such that $G = G_0\times G_1$, where $G_0$ is $\add(\mu,\zeta)$-generic, and $A \in V[G_0]$.
In $V[G_0]$, $\delta$ be such that $\zeta< \delta <\theta$ and $\{ X_\alpha \setminus \delta : \alpha\in A \}$ is pairwise disjoint.  
For any $p \in \add(\mu,\theta \setminus \zeta)$ and any $\eta \geq \delta$, there are $q \leq p$ and $\alpha,\beta \in A$ such that $q(\gamma_\eta^\alpha,0) \not= q(\gamma_\eta^\beta,0)$.  Since $G_1$ is generic, we have that for all $\eta \geq \delta$, there are $\alpha,\beta \in A$ such that $f(\alpha,\eta) \not= f(\beta,\eta)$.  Thus there is no $B \subseteq \theta$ of size $\theta$ such that $f$ is constant on $A \times B$.
\end{proof}

We can now prove Theorem \ref{dense}.  Suppose that in $V$, $I$ is a nonregular, $\kappa$-complete, normal ideal on $\p_\kappa(\lambda)$, where $\kappa = \mu^+$ and $\cf(\mu)$ is uncountable.  Let $\theta \geq \lambda^\mu$ be regular and such that $\theta^\mu = \theta$.  Let $G \subseteq \add(\omega,\theta)$ be generic.  By Lemma \ref{preservenonreg}, $\bar I$ is nonregular in $V[G]$.  Suppose $Z \subseteq \p_\kappa(\lambda)$ has cardinality $<\theta$.  Then there is $\zeta < \theta$ such that $G = G_0\times G_1$, where $G_0$ is $\add(\omega,\zeta)$-generic, and $Z \in V[G_0]$.  By Lemma \ref{nodense}, there is no $\lambda$-dense, $\kappa$-complete, normal ideal concentrating on $Z$ in $V[G]$.  For $\kappa$-complete normal ideals on $\p_\kappa(\lambda)$ in $V[G]$ that do not have any positive set of size $<\theta$, Lemma \ref{densepart} implies that if such an ideal were $\lambda$-dense, then we would have $\left( \begin{array}{c}
\lambda^+ \\ \theta
\end{array} \right)
\rightarrow
\left( \begin{array}{c}
\mu \\ \theta
\end{array} \right)_\mu,$
since $\lambda^\mu = \theta$ in $V[G]$.  
But Lemma \ref{killpart} implies that the weaker relation
$\left( \begin{array}{c}
\theta^+ \\ \theta
\end{array} \right)
\rightarrow
\left( \begin{array}{c}
\omega \\ \theta
\end{array} \right)_2$
fails in $V[G]$.  Thus $V[G]$ has no $\lambda$-dense, $\kappa$-complete, normal ideal on $\p_\kappa(\lambda)$.

\bibliographystyle{amsplain.bst}
\bibliography{masterbib.bib}
\end{document}